\documentclass[12pt]{article}
\usepackage{amsfonts,amsmath,amscd,amssymb,amsthm, graphics,color}

\newfont{\eufm}{eufm10 scaled\magstep1}

\newcommand{\cA}{\mathcal{A}}
\newcommand{\cC}{\mathcal{C}}

\newcommand{\cD}{\mathcal{D}}

\newcommand{\cV}{\mathcal{V}}

\newcommand{\cP}{\mathcal{P}}
\newcommand{\cF}{\mathcal{F}}

\newcommand{\cL}{\mathcal{L}}
\newcommand{\cG}{\mathcal{G}}

\newcommand{\cR}{\mathcal{R}}

\newcommand{\cX}{\mathcal{X}}

\newcommand{\cK}{\mathcal{K}}

\newcommand{\bbN}{\mathbb{N}}

\newcommand{\bbQ}{\mathbb{Q}}

\newcommand{\bbK}{\mathbb{K}}

\newcommand{\bbD}{\mathbb{D}}

\newtheorem{thm}{Theorem}[section]
\newtheorem{lem}[thm]{Lemma}
\newtheorem{cor}[thm]{Corollary}
\newtheorem{prop}[thm]{Proposition}
\newtheorem{defi}[thm]{Definition}
\newtheorem{rem}[thm]{Remark}

\newtheorem{alg}[thm]{Algorithm}
\def\para{\vspace{2mm}}
\def\id{{\rm ID}}
\def\dprim{\id{\rm prim}}

\def\dcont{\id{\rm cont}}

\def\dcres{\partial{\rm CRes}}
\def\ps{{\rm PS}}

\def\rank{{\rm rank}}
\def\ord{{\rm ord}}
\def\lead{{\rm lead}}
\def\prem{{\rm prem}}
\def\min{{\rm min}}

\def\gcd{{\rm gcd}}

\def\ker{{\rm Ker}}

\def\gcld{{\rm gcld}}
\def\c{{\rm c}}
\def\co{{\rm co}}
\def\l{{\rm l}}

\begin{document}
\title{A perturbed differential resultant based implicitization algorithm for linear DPPEs\thanks{
Supported by the Spanish ``Ministerio de Ciencia e Innovaci\' on"
under the Project MTM2008-04699-C03-01}}

\author{Sonia L. Rueda   \\
Dpto de Matem\' atica Aplicada, E.T.S. Arquitectura.\\
Universidad Polit\' ecnica de Madrid.\\
Avda. Juan de Herrera 4, 28040-Madrid, Spain.\\
{sonialuisa.rueda@upm.es}}

\date{}
\maketitle

\begin{abstract}
Let $\bbK$ be an ordinary differential field with derivation
$\partial$.
Let $\cP$ be a system of $n$ linear differential
polynomial parametric equations in $n-1$ differential parameters with implicit ideal $\id$. Given a nonzero linear differential polynomial $A$ in $\id$ we give necessary and sufficient conditions on $A$ for $\cP$ to be $n-1$ dimensional.
We prove the existence of a linear perturbation
$\cP_{\phi}$ of $\cP$ so that the linear complete differential resultant $\dcres_{\phi}$ associated to $\cP_{\phi}$ is nonzero. A nonzero linear differential polynomial in $\id$ is obtained from the lowest degree term of $\dcres_{\phi}$ and used to provide an implicitization algorithm for $\cP$.
\end{abstract}

%%%%%%%%%%%%%%%%%%%%%%%%%%%%%%%%%%%%%%%%%%%%%%%%%%%%%%%%%%%%%%%%%%%
\section{Introduction}
%%%%%%%%%%%%%%%%%%%%%%%%%%%%%%%%%%%%%%%%%%%%%%%%%%%%%%%%%%%%%%%%%%%
The use of algebraic elimination techniques such as Groebner basis and multivariate resultants to obtain the implicit equation of a unirational algebraic variety is well known (see for instance \cite{Cox}, \cite{Cox-2}). The development of similar techniques in the differential case is an active field of research. In \cite{Gao} characteristic set methods were used to solve the differential implicitization problem for differential rational parametric equations and alternative methods are emerging to treat special cases. In \cite{RS}, linear complete differential resultants were used to compute the implicit equation of a set of linear differential polynomial parametric equations (linear DPPEs). As in the algebraic case differential resultants often vanish under specialization and we are left with no candidate to be the implicit equation. This reason prevented us from giving an algorithm for differential implicitization in \cite{RS}. Motivated by Canny's method in \cite{Can} and its generalizations in \cite{DAE} and \cite{Roj} in the present work we consider a linear perturbation of a given system of linear DPPEs and use linear complete differential resultants to give a candidate to be the implicit equation of the system.

\para

Given a system $\cP(X,U)$ of $n$ linear ordinary differential polynomial parametric equations $x_1=P_1(U),\ldots ,x_n=P_n(U)$
in $n-1$ differential parameters $u_1,\ldots ,u_{n-1}$ (we give a precise statement of the problem in Section \ref{sec-Basic notions and notation}) we give an algorithm to decide if the dimension of the implicit ideal $\id$ of $\cP$ is $n-1$ and in the affirmative case provide the implicit equation of $\cP$.

\para

The linear complete differential resultant $\dcres (x_1-P_1(U),\ldots ,x_n-P_n(U))$ is the
algebraic resultant of Macaulay of a set of differential polynomials with $L$ elements.
It was defined in \cite{RS} as a generalization of Carra'Ferro's differential resultant \cite{CFproc} (in the linear case) in order to adjust the number $L$ of differential polynomials to the order of derivation of the variables $u_1,\ldots ,u_n$ in $F_i=x_i-P_i(U)$.

\para

In this paper, we provide a perturbation $\cP_{\phi} (X,U)$ of $\cP (X,U)$ so that the linear differential polynomials
$F_1-p\,\phi_1(U), \ldots, F_n-p\,\phi_n(U)$
have nonzero linear complete differential resultant $\dcres_{\phi}(p)$ which is a polynomial depending on $p$. It will be shown that the coefficient of the lowest degree term of $\dcres_{\phi}(p)$ is a nonzero linear differential polynomial that belongs to the implicit ideal $\id$ of $\cP (X,U)$. In fact, if $\dcres_{\phi}(p)$ has a nonzero independent term it equals $\dcres (F_1,\ldots ,F_n)$ and as proved in \cite{RS} it gives the implicit equation of $\cP (X,U)$.

\para

The main result of this paper generalizes the result previously mentioned from \cite{RS}. Given a nonzero linear differential polynomial $A$ in $\id$ necessary and sufficient conditions on $A$ are provided so that $A(X)=0$ is the implicit equation of $\cP(X,U)$. The higher order terms in the equations of $\cP(X,U)$ and the rank of the coefficient matrix of the set of $L$ polynomials used to construct the differential resultant $\dcres (F_1,\ldots ,F_n)$ play a significant role in this theory. The fact that we are dealing with linear differential polynomials will be also relevant, allowing us to treat then using differential operators.

\para

The paper is organized as follows. In Section \ref{sec-Basic
notions and notation} we introduce the main notions and notation.
Next we review the definition of the linear complete differential (homogeneous) resultant
in Section \ref{sec-Linear gamma-Differential Resultants}.
The main definitions regarding linear differential polynomials in $\id$ are given in Section \ref{sec-Implicit Ideal}. The next section contains the main result of the paper, namely a characterization of $\id$ in the $n-1$ dimensional case is provided in Section \ref{sec-dimension ID}. In Section \ref{sec-Linear perturbation} we give a perturbation $\cP_\phi(X,U)$ of $\cP(X,U)$ with nonzero differential resultant and use it to obtain a nonzero linear differential polynomial in $\id$ candidate to provide the implicit equation. Finally in Section \ref{sec-Implicitization algorithm} we give the implicitization algorithm and examples.

%%%%%%%%%%%%%%%%%%%%%%%%%%%%%%%%%%%%%%%%%%%%%%%%%%%%%%%%%%%%%%%%%%%
\section{Basic notions and notation}\label{sec-Basic notions and
notation}
%%%%%%%%%%%%%%%%%%%%%%%%%%%%%%%%%%%%%%%%%%%%%%%%%%%%%%%%%%%%%%%%%%%

In this section, we introduce the basic notions related to the
problem we deal with (as in \cite{RS08} and \cite{RS}), as well as notation and terminology used
throughout the paper. For further concepts and results on
differential algebra we refer to  \cite{Kol} and \cite{Ritt}.

\para

Let $\bbK$ be an ordinary differential field with derivation
$\partial$, ( e.g. $\bbQ(t)$, $\partial=\frac{\partial}{\partial
t}$). Let $X=\{x_1,\ldots ,x_n\}$ and $U=\{u_1,\ldots ,u_{n-1}\}$
be  sets of differential indeterminates over $\bbK$. Let $\bbN_0=\{0,1,2,\ldots ,n,\ldots\}$. For
$k\in\bbN_0$ we denote by $x_{ik}$ the $k$-th derivative of $x_i$, for $x_{i0}$ we simply write $x_i$.
Given a set $Y$ of differential indeterminates over $\bbK$ we
denote by $\{Y\}$ the set of derivatives of the elements of $Y$,
$\{Y\}=\{\partial^k y\mid y\in Y,\; k\in \bbN_0\}$, and  by $\bbK
\{X\}$ the ring of differential polynomials in the differential
indeterminates $x_1,\ldots ,x_n$, that is
\begin{displaymath}
\bbK\{X\}=\bbK[x_{ik}\mid i=1,\ldots ,n,\; k\in \bbN_0 ].
\end{displaymath}
Analogously for $\bbK \{U\}$.

\para

As defined in \cite{RS} we consider the system of linear DPPEs
\begin{equation}\label{DPPEs}
\cP(X,U) = \left\{\begin{array}{ccc}x_1 &= & P_1 (U)\\ & \vdots  & \\
x_n&= & P_n (U).\end{array}\right.
\end{equation}
where $P_1,\ldots ,P_n\in \bbK \{U\}$ with degree at most $1$ and not all $P_i\in \bbK$, $i=1,\ldots ,n$.
There exists $a_i\in\bbK$ and an homogeneous differential polynomial $H_i\in \bbK \{U\}$ such that
\[
F_i(X,U)=x_i-P_i(U)= x_i-a_i+H_{i}(U).
\]

Given $P\in \bbK \{X\cup U\}$ and $y\in X\cup U$, we denote by $\ord(P,y)$ the {\sf
order} of $P$ in the variable $y$. If $P$ does not have
a term in $y$ then we define $\ord(P,y)=-1$. To ensure that the number of parameters is $n-1$, we assume that
for each $j\in\{1,\ldots ,n-1\}$ there exists $i\in\{1,\ldots
,n\}$ such that $\ord (F_i, u_j)\geq 0$.

\para

The {\sf implicit ideal} of the system \eqref{DPPEs} is the differential prime ideal
\begin{equation*}
\id =\{f\in \bbK \{X\}\mid f(P_1 (U),\ldots ,P_n
(U))=0\}.
\end{equation*}
Given a characteristic set $\cC$ of $\id$ then $n-\mid \cC \mid$ is the (differential)
dimension of $\id$, by abuse of notation, we will also speak about
the dimension of a DPPEs system meaning the dimension of its
implicit ideal.

\para

If $\dim(\id)=n-1$, then $\cC=\{A(X)\}$ for some irreducible
differential polynomial $A\in\bbK\{X\}$. The polynomial $A$  is
called a {\sf characteristic polynomial} of $\id$.
The {\sf implicit equation} of a $(n-1)$-dimensional
system of DPPEs, in $n$ differential indeterminates $X=\{x_1,\ldots,x_n\}$,
is defined as the equation
$A(X)=0$, where $A$ is any characteristic polynomial of the
implicit ideal $\id$ of the system.

\para

Let $\bbK [\partial]$ be the ring  of differential operators with
coefficients in $\bbK$. If $\bbK$ is not a field of constants with respect to $\partial$,
then $\bbK [\partial]$ is not commutative but
$\partial k-k\partial=\partial (k)$
for all $k\in \bbK$. The ring $\bbK [\partial]$ of differential
operators with coefficients in $\bbK$ is left euclidean (and also
right euclidean). Given $\cL,\cL'\in\bbK [\partial]$, by applying
the left division algorithm we obtain $q,r\in\bbK [\partial]$,
the left quotient and the left reminder of $\cL$ and $\cL'$
respectively, such that $\cL=\cL' q+r$ where $\deg
(r)<\deg(\cL')$.

%%%%%%%%%%%%%%%%%%%%%%%%%%%%%%%%%%%%%%%%%%%%%%%%%%%%%%%%%%%%%%%%%%%
\section{Linear complete differential resultants}\label{sec-Linear gamma-Differential Resultants}
%%%%%%%%%%%%%%%%%%%%%%%%%%%%%%%%%%%%%%%%%%%%%%%%%%%%%%%%%%%%%%%%%%%

We review next the results on linear complete differential resultants from \cite{RS} that
will be used in this paper.

\para

Let $\bbD$ be a differential integral domain, and let
$f_i\in\bbD\{U\}$ be a linear ordinary differential polynomial of
order $o_i$, $i=1,\ldots ,n$.
For each $j\in \{1,\ldots ,n-1\}$ we define the positive integers
\begin{align*}
&\gamma_j (f_1,\ldots ,f_n):=\min\{o_i-\ord (f_i,u_j)\mid i\in\{1,\ldots ,n\}\},\\
&\gamma (f_1,\ldots ,f_n):=\sum_{j=1}^{n-1} \gamma_j(f_1,\ldots ,f_n).
\end{align*}
Let $N=\sum_{i=1}^n o_i$ then the {\sf completeness index} $\gamma (f_1,\ldots ,f_n)$
verifies $\gamma (f_1,\ldots ,f_n)\leq N-o_i$, for all $i\in\{1,\ldots ,n\}$.

The {\sf linear complete differential resultant}
$\dcres (f_1,\ldots,f_n)$ is the Macaulay's algebraic
resultant of the differential polynomial set
\begin{align*}
&\ps(f_1,\ldots ,f_n):=\\
&\{\partial^{N-o_i-\gamma} f_i,\ldots ,\partial f_i, f_i \mid
i=1,\ldots ,n,\,\,\, \gamma =\gamma(f_1,\ldots ,f_n)\}.
\end{align*}
The set $\ps (f_1,\ldots ,f_n)$ contains
$L=\sum_{i=1}^n (N-o_i-\gamma+1)$ polynomials in the following set
$\cV$ of $L-1$ differential variables
\begin{equation*}\label{eq-V}
\cV =\{u_j,u_{j1}\ldots ,u_{j N-\gamma_j-\gamma} \mid \gamma_j=\gamma_j
(f_1,\ldots ,f_n), j=1,\ldots ,n-1\}.
\end{equation*}

\para

Let $h_i\in\bbD\{U\}$ be a linear ordinary
differential homogeneous polynomial of order $o_i$, $i=1,\ldots ,n$ with $N=\sum_{i=1}^no_i\geq 1$.
We define the differential polynomial set
\begin{align*}
&\ps^h(h_1,\ldots ,h_n):=\{\partial^{N-o_i-\gamma -1}h_i,\ldots ,\partial h_i, h_i \mid\\
&i\in\{1,\ldots ,n\},\,\,\,N-o_i-\gamma-1\geq 0,\,\,\, \gamma =\gamma
(h_1,\ldots ,h_n)\}.
\end{align*}
Observe that $N\geq 1$ implies $\ps^h(h_1,\ldots ,h_n)\neq\emptyset$.
The {\sf linear complete differential homogenous resultant}
$\dcres^h (h_1,\ldots ,h_n)$ is the Macaulay's algebraic resultant of the set $\ps^h(h_1,\ldots ,h_n)$.
The set $\ps^h(h_1,\ldots ,h_n)$ contains
$L^h=\sum_{i=1}^n (N-o_i-\gamma)$ polynomials in the set
$\cV^h$ of $L^h$ differential variables
\begin{align*}
\cV^h =\{u_j,u_{j1}\ldots ,u_{j N-\gamma_j-\gamma -1}\mid
\gamma_j=\gamma_j (h_1,\ldots ,h_n), j=1,\ldots ,n-1\}.
\end{align*}

\para

We review next the matrices that will allow the use of
determinants to compute $\dcres (f_1,\ldots ,f_n)$ and $\dcres^h
(h_1,\ldots ,h_n)$. The order $u_1<\cdots <u_{n-1}$ induces an orderly ranking on $U$ (i.e. and order on $\{U\}$) as follows
(see \cite{Kol}, page 75): $u<\partial u$, $u<u^\star\Rightarrow
\partial u <\partial u^\star$ and $k<k^\star\Rightarrow\partial^k u <
\partial^{k^\star} u^\star$, for all $u, u^\star\in U$, $k,k^\star\in\bbN_0$. We
set $1<u_1$.

For $i=1,\ldots ,n$, $\gamma=\gamma
(f_1,\ldots ,f_n)$ and $k=0,\ldots ,N-o_i-\gamma$
define de positive integer
$l(i,k)=(i-1)(N-\gamma)-\sum_{h=1}^{i-1}o_i+i+k$ in $\{1,\ldots ,L\}$.
The {\sf complete differential resultant matrix}  $M(L)$ is the $L\times L$ matrix containing the
coefficients of $\partial^{N-o_i-\gamma -k}f_i$ as a polynomial in $\bbD[\cV]$
in the $l(i,k)$-th row, where the coefficients are written in
decreasing order with respect to the orderly ranking on $U$.
In this situation:
\begin{align*}
\dcres (f_1,\ldots ,f_n)&=\det(M(L)).
\end{align*}

If $N\geq 1$, for $\gamma=\gamma (h_1,\ldots ,h_n)$, $i\in \{1,\ldots ,n\}$ such that $N-o_i-\gamma-1\geq 0$ and
$k=0,\ldots ,N-o_i-\gamma-1$ define the positive integer $l^h(i,k)=(i-1)(N-\gamma-1)-\sum_{h=1}^{i-1}o_i+i+k$
in $\{1,\ldots ,L^h\}$. The {\sf complete differential homogeneous
resultant matrix} $M(L^h)$ is the $L^h\times L^h$ matrix containing the
coefficients of $\partial^{N-o_i-\gamma -k-1} h_i$ as a polynomial in $\bbD[\cV^h]$
in the $l^h(i,k)$-th row, where the coefficients are written in
decreasing order with respect to the orderly ranking on $U$.
In this situation:
\begin{align*}
\dcres^h (h_1,\ldots ,h_n)&=\det(M(L^h)).
\end{align*}

Throughout the remaining parts of the paper we will say differential ( homogeneous) resultant always meaning linear complete differential ( homogeneous) resultant.

\subsection{Linear complete differential resultants from linear DPPEs}\label{subsec-Linear gamma-Differential Resultants}

We highlight in this section some facts on differential resultants of
the differential polynomials $F_i$ and $H_i$ obtained from a system of linear DPPEs as
in Section \ref{sec-Basic notions and notation}.

\para

Let $\gamma=\gamma (F_1,\ldots ,F_n)=\gamma (H_1,\ldots
,H_n)$ and $\bbD=\bbK\{X\}$. Then $\dcres (F_1,\ldots ,F_n)$ and $\dcres^h (H_1,\ldots ,H_n)$ are closely related
as shown in \cite{RS}, Section 5.
Since $F_i(X,U)= x_i-a_i+H_{i}(U)$, if $N\geq 1$ the matrix $M(L^h)$ is a submatrix of $M(L)$ obtained removing $n$ specific rows and columns.
This fact together with the identities below allowed to prove that (when $N\geq 1$)
\[\dcres (F_1,\ldots ,F_n)=0 \Leftrightarrow \dcres^h (H_1,\ldots ,H_n)=0.\]
The next matrices will play an important role in the remaining parts of the paper.
\begin{itemize}
\item Let $S$ be the $n\times (n-1)$ matrix whose entry $(i,j)$ is the coefficient of
$u_{n-j\, o_i-\gamma_j}$ in $F_i$, $i\in\{1,\ldots ,n\}$,
$j\in\{1,\ldots ,n-1\}$. We call $S$ the {\sf leading matrix} of $\cP(X,U)$.
For $i\in\{1,\ldots n\}$ let $S_i$ be
the $(n-1)\times (n-1)$ matrix obtained by removing the $i$-th row of $S$.

\item Let $M_{L-1}$ be the $L\times (L-1)$ principal submatrix of
$M(L)$. We call $M_{L-1}$ the {\sf principal matrix} of $\cP(X,U)$.
\end{itemize}

\para

Let $\cX=\{x_i,x_{i 1},\ldots ,x_{i\, N-o_i-\gamma}\mid i=1,\ldots ,n\}$. Given $x\in \cX$, say
$x=x_{ik}$ with $k\in\{0,1,\ldots ,N-o_i-\gamma\}$, let us call $M_x$ the
submatrix of $M_{L-1}$ obtained by removing the row corresponding
to the coefficients of
$\partial^{k}F_i=x_{ik}+\partial^k(H_i(U)-a_i)$. Then, developing
the determinant of $M(L)$ by the last column we obtain
\begin{equation}\label{eq-dResMx}
\dcres (F_1,\ldots ,F_n)=\sum_{i=1}^n \sum_{k=0}^{N-o_i-\gamma} b_{ik} \det (M_{x_{ik}}) (x_{ik}-\partial^k a_i),
\end{equation}
with $b_{ik}=\pm 1$ according to the row index of $x_{ik}-\partial^k a_i$ in the matrix $M(L)$. Also,
for every $i\in\{1,\ldots ,n\}$, there exists $a\in\bbN$ such
that
\begin{equation}\label{eq-Mxi}
\det(M_{x_{i N-o_i-\gamma}})=(-1)^a \dcres^h (H_1,\ldots ,H_n) \det(S_i).
\end{equation}

%%%%%%%%%%%%%%%%%%%%%%%%%%%%%%%%%%%%%%%%%%%%%%%%%%%%%%%%%%%%%%%%%%%
\section{The implicit ideal $\id$}\label{sec-Implicit Ideal}
%%%%%%%%%%%%%%%%%%%%%%%%%%%%%%%%%%%%%%%%%%%%%%%%%%%%%%%%%%%%%%%%%%%

Let $\cP(X,U)$, $F_i$, $H_i$ be as in Section \ref{sec-Basic
notions and notation}.  Let $\ps=\ps(F_1,\ldots,F_n)$ and let $\id$ be the implicit ideal of
$\cP(X,U)$. In this section, we  review the computation of $\id$ in terms of characteristic sets (see \cite{Gao} and \cite{RS}) and give some definitions related with linear differential polynomials in $\id$ that will be important in the remaining parts of the paper.

\para

Let $\cX$ and $\cV$ be as in Section \ref{subsec-Linear gamma-Differential Resultants} and observe that $\ps\subset \bbK[\cX][\cV]$. Let $(\ps)$ be the ideal generated by $\ps$ in $\bbK[\cX][\cV]$ and let $[\ps]$ be the differential ideal generated by $\ps$ in $\bbK\{X\}$.

The order $x_n<\cdots <x_1$ induces a ranking on $X$ as follows (see \cite{Kol}, page 75):
$x<\partial x$ and $x<x^\star\Rightarrow \partial^k x < \partial^{k^\star}
x^\star$, for all $x, x^\star\in X$, $k,k^\star\in\bbN_0$.
\begin{itemize}
\item We call $\cR$ the ranking on $X\cup U$ that eliminates $X$ with
respect to $U$, that is $\partial^k x> \partial^{k^\star} u$, for all
$x\in X$, $u\in U$ and $k,k^\star\in\bbN_0$.

\item We call $\cR^\star$ the ranking on $X\cup U$ that eliminates $U$ with respect to $X$,
that is $\partial^k x< \partial^{k^\star} u$, for all $x\in X$, $u\in U$ and $k,k^\star\in\bbN_0$.
\end{itemize}

\para

Note that, because
of the particular structure of $F_i$, with respect to $\cR$ then $\ps$ is a chain (see \cite{Ritt}, page 3) of $L$
differential polynomials, with
$L$ as in Section \ref{sec-Linear gamma-Differential Resultants}.
Let $\cA$ be a characteristic set of $[\ps]$ and $\cA_0=\cA\cap\bbK\{X\}$.
By \cite{Gao} then
\[\id=[\ps]\cap \bbK\{X\}=[\cA_0].\]
To compute a characteristic set of $[\ps]$ we will use the reduced
Groebner basis of $(\ps)$ with respect to
lex monomial order induced by the ranking $\cR^\star$. We call $\cG$ the
{\sf Groebner basis associated to the system} $\cP (X,U)$.
For that purpose we apply the algorithm given in
\cite{BL}, Theorem 6, that we briefly include below for completion.

Given $P\in\bbK\{X\cup U\}$, the {\sf lead} of $P$ is the highest derivative
present in $P$ w.r.t. $\cR^\star$, we denote it by $\lead (P)$.
Given $P,Q\in \bbK\{X\cup U\}$ we denote by $\prem (P,Q)$ the
{\sf pseudo-remainder} of $P$ with respect to $Q$, \cite{Ritt}, page 7.
Given a chain $\cA=\{A_1,\ldots ,A_t\}$ of elements of $\bbK\{X\cup U\}$ then
$\prem (P,\cA)=\prem (\prem (P,A_t), \{A_1,\ldots
,A_{t-1}\})$ and $\prem (P,\emptyset)=P$.

\begin{alg}\label{alg-characteristic set}
{\sf Given} the set of differential polynomials $\ps$ the next algorithm {\sf
returns} a characteristic set $\cA$ of $[\ps]$.
\begin{enumerate}
\item Compute the reduced Groebner basis $\cG$ of $(\ps)$ with respect to
lex monomial order induced by $\cR^\star$.

\item Assume that the elements of $\cG$ are arranged in increasing
order $B_0<B_1<\cdots <B_{L-1}$ w.r.t. $\cR^\star$. Let $\cA^0=\{B_0\}$.
For $i$ from $1$ to $L-1$ do, if \lead$(B_i)\neq$\lead$(B_{i-1})$
then $\cA^{i}:=\cA^{i-1} \cup\{\mbox{\prem}(B_i,\cA^{i-1})\}$. $\cA:=\cA^{L-1}$.
\end{enumerate}
\end{alg}

\para

We are dealing with a linear system of polynomials and
computing a Groebner basis is equivalent to performing gaussian elimination.
Some details on this computation were given in \cite{RS} and we include them
below to be used further in this paper.
Let $M_{2L}$ be the $L\times (2L)$ matrix whose $k$-th row
contains the coefficients, as a polynomial in $\bbK[\cX][\cV]$, of the $(L-k+1)$-th polynomial in $\ps$ and where the
coefficients are written in decreasing order w.r.t. $\cR^\star$.

\[M_{2L}=\left[\begin{array}{cc}M_{L-1}&
\begin{array}{cccccccc}1& & & & & & &\partial^{N-o_1-\gamma} a_1\\ &\ddots & & & & & &\vdots \\ & & 1& & & & & a_1 \\  & & &\ddots & & & &\vdots \\ & & & & 1 & & &\partial^{N-o_n-\gamma} a_n\\ & & & & &\ddots & &\vdots \\  & & & & & & 1 &a_n \end{array}
\end{array}\right].\]

The polynomials corresponding to the rows of the reduced echelon form
$E_{2L}$ of $M_{2L}$ are the elements of the Groebner basis associated to $\cP(X,U)$. Let $\cG_0=\cG\cap \bbK\{X\}$. From \cite{RS}, Lemma 20(1):
\begin{lem}\label{lem-G0rev}
The Groebner basis associated to the system $\cP (X,U)$ is a set of linear
differential polynomials  $\cG=\{B_0,B_1,\ldots ,B_{L-1}\}$ where $B_0<B_1<\cdots <B_{L-1}$
with respect to the ranking $\cR^\star$ and $B_0\in\cG_0$. Hence $(\ps)\cap \bbK\{X\}=(\cG_0)$ is nonzero.
\end{lem}

%%%%%%%%%%%%%%%%%%%%%%%%%%%%%%%%%%%%%%%%%%%%%%%%%%%%%%%%%%%%%%%%%%%
\subsection{On linear differential polynomials in $\id$}\label{subsec-diffpolsPS}
%%%%%%%%%%%%%%%%%%%%%%%%%%%%%%%%%%%%%%%%%%%%%%%%%%%%%%%%%%%%%%%%%%%

\para

The next definitions will play an important role throughout the paper.
Given a nonzero linear differential polynomial $B$ in $[\ps ]$ there exist differential operators $\cF_i\in\bbK[\partial]$, $i=1,\ldots ,n$ such that
\[B(X,U)=\sum_{i=1}^{n} \cF_i (F_i(X,U)).\]
If $B$ belongs to $\id =[\ps ]\cap\bbK\{X\}$ then
\begin{equation*}\label{eq_AD}
B=\sum_{i=1}^n \cF_i(x_i-a_i)\mbox{ and }\sum_{i=1}^n \cF_i(H_i(U))=0.
\end{equation*}
If $B$ belongs to $(\ps)$ then $\deg(\cF_i)\leq N-o_i-\gamma$, $i=1,\ldots ,n$.

\begin{defi}
Given a nonzero linear differential polynomial $B$ in $(PS)$ with $B(X,U)=\sum_{i=1}^{n} \cF_i (F_i(X,U))$, $\cF_i\in\bbK[\partial]$.
\begin{enumerate}
\item We define the {\sf co-order} of $B$ in $(\ps )$ as the highest positive integer $\c (B)$ such that $\partial^{\c (B)} B\in (\ps)$. Observe that
\[\c (B)=\min\{N-o_i-\gamma-\ord(B,x_i)\mid i=1,\ldots ,n\}.\]

\item For $i\in\{1,\ldots ,n\}$ let $\alpha_i$ be the coefficient of $\partial^{N-o_i-\gamma-\c (B)}$ in $\cF_i$. We call $(\alpha_1,\ldots ,\alpha_n)$ the {\sf leading coefficients vector} of $B$ in $(\ps )$ and we denote it by $\l (B)$.
\end{enumerate}
\end{defi}

Let $S$ be the leading matrix of the system $\cP(X,U)$. Denote by $S^T$ the transpose matrix of $S$.

\begin{rem}\label{rem-kerSt}
The $i$-th row of $S$ equals $\l (F_i(X,U))$, $i\in \{1,\ldots ,n\}$.
Given a nonzero linear $B\in (PS)$, the $i$-th row of $S$ also consists of the coefficients of $u_{n-j\, N-\gamma_j-\gamma-\c (B)}$, $j\in \{1,\ldots
,n-1\}$ in $\partial^{N-o_i-\gamma-\c (B)}F_i (X,U)$.
Therefore if $\ord(B,u_j)<N-\gamma_j-\gamma-\c (B)$ then $\l (B)S=0$, that is $\l (B)\in\ker (S^T)$.
\end{rem}

\begin{defi}
Given a nonzero linear differential polynomial $B$ in $\id$ with $B=\sum_{i=1}^{n} \cF_i (x_i-a_i)$, $\cF_i\in\bbK[\partial]$.
\begin{enumerate}
\item We define the {\sf $\id$-content} of $B$ as the greatest common left
divisor of $\cF_{1},\ldots ,\cF_{n}$ (we write $\gcld (\cF_1,\ldots ,\cF_n)$). We denote it by $\dcont (B)$.

\item There exist $\cL_{i}\in\bbK[\partial]$ such that $\cF_{i}=\dcont (B)\cL_{i}$,
$i=1,\ldots ,n$ and $\cL_{1},\ldots ,\cL_{n}$ are coprime (we write $(\cL_1,\ldots ,\cL_n)=1$). We define the {\sf $\id $-primitive part} of $B$ as
\[\dprim (B)(X,U)=\sum_{i=1}^n \cL_i (x_i-a_i).\]

\item If $\dcont (B)\in\bbK$ then we day that $B$ is {\sf $\id $-primitive} .
\end{enumerate}
\end{defi}

If $A=\dprim (B)=\sum_{i=1}^{n} \cL_i (x_i-a_i)$ then $\c (A)\geq \deg(\dcont (B))$ and $\deg(\cL_i)\leq N-o_i-\gamma-\c (A)$, $i=1,\ldots ,n$.

\begin{lem}
Given a linear differential polynomial $B\in (\ps)\cap\bbK\{X\}$ then $\dprim (B)\in(\ps)\cap\bbK\{X\}$.
\end{lem}
\begin{proof}
For $i=1,\ldots n$ and $j=1,\ldots n-1$,
there exist differential operators $\cL_{ij}\in \bbK [\partial]$
such that
$H_i(U)=\sum_{j=1}^{n-1} \cL_{ij} (u_j)$. If $B(X,U)=\sum_{i=1}^{n} \cF_i (x_i-a_i)$
then $\sum_{i=1}^n \cF_i(H_i(U))=0$. As a consequence, for each $j\in\{1,\ldots ,n-1\}$ then $\sum_{i=1}^n \cF_i(\cL_{ij}(u_j))=0$.
Let $\cL=\dcont (B)$ then $\cF_i= \cL\cL_{i}$ with $\cL_i\in\bbK[\partial]$ and $\dprim (B)=\sum_{i=1}^n \cL_i (x_i-a_i)$.
Thus $\cL \sum_{i=1}^n \cL_i \cL_{ij}=0$ and $\cL\neq 0$ so the differential operator
$\sum_{i=1}^n \cL_i \cL_{ij}=0$. We conclude that $\sum_{i=1}^n \cL_i(H_i(U))=0$.
Therefore $\dprim (B)\in \bbK\{X\}$ which proofs the lemma.
\end{proof}

%%%%%%%%%%%%%%%%%%%%%%%%%%%%%%%%%%%%%%%%%%%%%%%%%%%%%%%%%%%%%%%%%%%
\section{Conditions for $\dim (\id)=n-1$}\label{sec-dimension ID}
%%%%%%%%%%%%%%%%%%%%%%%%%%%%%%%%%%%%%%%%%%%%%%%%%%%%%%%%%%%%%%%%%%%

Let $\cP(X,U)$, $F_i$, $H_i$ be as in Section \ref{sec-Basic
notions and notation}.  Let $\ps=\ps(F_1,\ldots,F_n)$ and let $\id$ be the implicit ideal of
$\cP(X,U)$. Let $S$ and $M_{L-1}$ be the leading and principal matrices of $\cP(X,U)$ respectively, as defined in Section \ref{subsec-Linear gamma-Differential Resultants}. Let $\cG$ be the Groebner basis associated to the system $\cP (X,U)$,  $\cG_0=\cG\cap \bbK\{X\}$ and denote by $|\cG_0|$ the number of elements of $\cG_0$.

By Lemma \ref{lem-G0} the ideal $(\ps)\cap\bbK\{X\}=(\cG_0)$ is nonzero.
 Given a nonzero $\id$-primitive differential polynomial $A$ in $(\ps)\cap\bbK\{X\}$ and co-order $\c (A)$ as defined in Section \ref{subsec-diffpolsPS}, in this section we provide necessary and sufficient conditions on $S$, $M(L)$ and $A$  for $A(X)=0$ to be the implicit equation of $\cP(X,U)$.

%%%%%%%%%%%%%%%%%%%%%%%%%%%%%%%%%%%%%%%%%%%%%%%%%%%%%%%%%%%%%%%%%%%
\subsection{Necessary conditions for $\dim (\id)=n-1$}\label{sec-nec dimension ID}
%%%%%%%%%%%%%%%%%%%%%%%%%%%%%%%%%%%%%%%%%%%%%%%%%%%%%%%%%%%%%%%%%%%

For $N=\sum_{i=1}^n o_i=0$ then $\cP (X,U)$ in a system of $n$ linear equations in $n-1$ indeterminates.

\begin{lem}\label{lem-N0}
If $N=0$ then $\dim \id=n-1$ if and only if $\rank(S)=n-1$.
\end{lem}
\begin{proof}
The matrix $M(L)$ is the $n\times n$ matrix whose $n\times (n-1)$ principal submatrix is $S$ and whose last column contains $x_i-a_i$ in the $i$-th row, $i=1,\ldots ,n$.
Then for linear $U_i\in\bbK\{X\}$
\[\rank(S)=n-1\Leftrightarrow\cG=\{B_0,u_1-U_1(X),\ldots ,u_n-U_n(X)\}.\]
Equivalently $\{B_0\}$ is a characteristic set of $\id$.
\end{proof}

We prove in the next theorem that one statement is true in general.

\begin{thm}\label{thm-Sdim}
Let $\cP (X,U)$ be a system of linear DPPEs with implicit ideal $\id$ and leading matrix $S$.
If $\dim\id=n-1$ then $\rank (S)=n-1$.
\end{thm}
\begin{proof}
For $N=0$ the result was proved in Lemma \ref{lem-N0}.
Let us suppose that $N\geq 1$. Let $\overline{\ps}=\{\partial^{N-o_i-\gamma-1} F_i,\ldots ,\partial F_i, F_i \mid i=1,\ldots ,n \}$. Let $\cV^h$ as defined
in  Section \ref{sec-Linear gamma-Differential Resultants} and let
$\overline{\cX}=\{x_i,\ldots ,x_{i N-o_i-\gamma-1}\mid i=1,\ldots n\}$.
Let $M_{2L^h}$ be the $L^h\times (2L^h+1)$ matrix whose $k$-th row
contains the coefficients, as a polynomial in $\bbK[\overline{\cX}][\cV^h]$,
of the $(L^h-k+1)$-th polynomial in $\overline{\ps}$ and where the
coefficients are written in decreasing order w.r.t. $\cR^\star$.
Then $\rank (M_{2L^h})=L^h$ since $M_{2L^h}$ has the $L^h\times L^h$ identity matrix as a submatrix.

Since $\rank (S)\leq n-1$ then $\ker (S^T)\neq \{\overline{0}\}$.
Given $\alpha=(\alpha_1,\ldots ,\alpha_n)\in \ker (S^T)$, $\alpha\neq \overline{0}$ then
\[B(X,U)=\sum_{i=1}^n \alpha_i \partial^{N-o_i-\gamma}F_i(X,U)-\sum_{i=1}^n \alpha_i \partial^{N-o_i-\gamma}(x_i-a_i)\in \bbK[\overline{\cX}][\cV^h].\]
Since $\rank (M_{2L^h})=L^h$ and the first $L^h$ columns of $M_{2L^h}$ correspond to the columns indexed by $\cV^h$ then there exists $\cG_i\in\bbK[\partial]$ with $\deg (\cG_i)\leq N-o_i-\gamma-1$, $i=1,\ldots ,n$ such that
\[B(X,U)+\sum_{i=1}^n \cG_i(F_i(X,U))\in \bbK[\overline{\cX}].\]
Let
\[A_{\alpha}=\sum_{i=1}^n \alpha_i \partial^{N-o_i-\gamma}F_i(X,U)+\sum_{i=1}^n \cG_i(F_i(X,U)).\]
Then $A_{\alpha}\in (PS)\cap \bbK\{X\}$ has leading vector $l(A_{\alpha})=\alpha$ and co-order $\c (A_{\alpha})=0$.

 Since $\dim\id=n-1$ then $A_{\alpha}\in \id=[A]$ for some differential polynomial $A\in\bbK\{X\}$ which is linear by Lemma \ref{lem-G0}. There exists $\cD_{\alpha}\in\bbK[\partial]$ with $\deg(\cD_{\alpha})=\c (A)$ such that $A_{\alpha}= \cD_{\alpha} (A)$. Then there is a nonzero $c_{\alpha}\in\bbK$ such that $l(A_{\alpha})=c_{\alpha} l(A)$.

 Given $\beta\in \ker (S^T)$, $\overline{0}\neq \beta\neq \alpha$ there exist $A_{\beta}\in \id=[A]$ with $l(A_{\beta})=\beta$, $\c (A_{\beta})=0$ and a nonzero $c_{\beta}\in\bbK$ such that $l(A_{\beta})=c_{\beta} l(A)$.
 Hence $\beta=  (c_{\beta}/c_{\alpha}) \alpha$ which proves that the
dimension of $\ker (S^T)$ is equal to $1$. Equivalently, $\rank (S)=n-1$.
\end{proof}

As expected, for $N>0$ the next results show that we need some more requirements for $\dim (\id)=n-1$.

\begin{lem}\label{lem-G0}
\begin{enumerate}
\item The number of elements of $\cG_0$ equals  $|\cG_0|=L-\rank (M_{L-1})$.

\item For every nonzero linear $B\in (\cG_0)$ then $|\cG_0|\geq \c  (B)+1$.
\end{enumerate}
\end{lem}
\begin{proof}
\begin{enumerate}
\item Let $M_{2L}$ be the $L\times 2L$ matrix defined in Section \ref{sec-Implicit Ideal} and $E_{2L}$ its redured echelon form.
The number of elements of $\cG_0$ is the number of rows in
$E_{2L}$ with zeros in the first $L-1$ columns. Thus
$|\cG_0|=L-\rank (M_{L-1})$.

\item Given $B\in (\cG_0)=(\ps)\cap\bbK\{X\}$. By definition of  $\c (B)$ then $\partial B$,$\ldots$ , $\partial^{\c (B)} B\in (\ps)\cap\bbK\{X\}$. Also, there exists $k\in\{1,\ldots ,n\}$ such that $\ord(B,x_k)=N-o_k-\gamma-\c (B)$, we can assume that the coefficient of $x_{k N-o_k-\gamma}$ in $B$ is $1$. 
Thus $M_{2L}$ is row equivalent to an $L\times 2L$
matrix with $\partial^{\c (B)} B,\ldots ,\partial B ,B$ in the last $\c (B)+1$
rows. Namely,  replace the row of $M_{2L}$ corresponding to the
coefficients of $\partial^{N-o_k-\gamma-t}F_k$ by $\partial^{\c (B)-t} B$, $t=0,\ldots ,\c (B)$ and next reorder the rows of the obtained matrix. Therefore $|\cG_0|\geq \c (B)+1$.
\end{enumerate}
\end{proof}

If the dimension of $\id$ is $n-1$ then $\id =[A]$ for some $A\in(\ps)\cap\bbK\{X\}=(\cG_0)$ and $\{A\}$ is a characteristic set of  $\id$.
By Lemma \ref{lem-G0} then $A$ is linear and we give some more requirements for $A$ in the next theorem.

\begin{thm}\label{thm-nec-coA}
Let $\cG$ be the Groebner basis associated to the system $\cP (X,U)$ with implicit ideal $\id$, $\cG_0=\cG\cap \bbK\{X\}$. If $\dim\id=n-1$ then $\id=[A]$ where $A$ is a nonzero linear differential polynomial verifying.
\begin{enumerate}
\item $A$ is an $\id$-primitive differential polynomial in $(\cG_0)$.

\item $|\cG_0|= \c (A)+1$.
\end{enumerate}
\end{thm}
\begin{proof}
\begin{enumerate}
\item Let $\cL=\dcont(A)$ and $A'=\dprim (A)$. If $A$ is not $\id$-primitive then $\deg (\cL)\geq 1$ and $A=\cL (A')$ contradicting that $\{A\}$ is a characteristic set of $\id$.

\item If $\dim\id =n-1$ and $|\cG_0|=m+1$ then $\cG_0=\{B_0,\cD_1 (B_0),\ldots ,\cD_m (B_0) \}$ with $\cD_i\in\bbK [\partial]$, $\deg (\cD_1)> 0$, $\deg (\cD_{i+1})> \det(\cD_{i})$, $i=1,\ldots ,m$. Then $m\leq \c (B_0)$. Thus by Lemma \ref{lem-G0} then $|\cG_0|= \co (B_0)+1$.
On the other hand $\id=[B_0]$ and $A\in \id$ then $A=\cD (B_0)$ for some $\cD\in \bbK [\partial]$ but we just proved that $A$ is $\id$-primitive then $\cD\in \bbK$. This implies $|\cG_0|= \c (A)+1$.
\end{enumerate}
\end{proof}

Observe that, if $\dim\id=n-1$, given $A$ and $B$ nonzero linear $\id$-primitive differential polynomials in $(\cG_0)$ then $\id=[A]=[B]$ and $\c (A)=\c (B)$.

\begin{cor}\label{cor-nec-coA}
Let $\cG$ be the Groebner basis associated to the system $\cP (X,U)$ with implicit ideal $\id$, $\cG_0=\cG\cap \bbK\{X\}$. If $\dim\id=n-1$, for every
nonzero linear $\id$-primitive differential polynomial $A$ in $(\cG_0)$ then $\id=[A]$ and $|\cG_0|= \c (A)+1$.
\end{cor}

%%%%%%%%%%%%%%%%%%%%%%%%%%%%%%%%%%%%%%%%%%%%%%%%%%%%%%%%%%%%%%%%%%%
\subsection{Sufficient conditions for $\dim (\id)=n-1$}\label{sec-suf dimension ID}
%%%%%%%%%%%%%%%%%%%%%%%%%%%%%%%%%%%%%%%%%%%%%%%%%%%%%%%%%%%%%%%%%%%

Given a nonzero linear $\id$-primitive differential polynomial $A$ in $(\ps)\cap\bbK\{X\}$ we give sufficient conditions for $A$ to be a characteristic polynomial of $\id$.

\begin{lem}\label{lem-G0char}
Let $\cG$ be the Groebner basis associated to the system $\cP (X,U)$ with implicit ideal $\id$, $\cG_0=\cG\cap \bbK\{X\}$.  Let $S$ be the principal matrix of $\cP (X,U)$. Given a nonzero linear $\id$-primitive differential polynomial $A$ in $(\cG_0)$ with $|\cG_0|= \c (A)+1$
the following statements hold.
\begin{enumerate}
\item For $j=0,1,\ldots ,\c (A)$ there exist $\cD_j\in\bbK [\partial]$ with $\deg (\cD_j)=j$ such that $\cG_0=\{B_0=\cD_0 (A),B_1=\cD_1(A),\ldots ,B_{\c (A)}=\cD_{\c (A)}(A)\}$.

\item If $\rank (S)=n-1$. Given $B\in \cG\backslash\cG_0$, let us suppose there exists a positive integer $e_B$ such that $1\leq e_B\leq \c (A)$ and $\ord(B,u_j)\leq N-\gamma_j-\gamma-e_B$, $j=1,\ldots ,n-1$. Then there exists $\overline{B}\in (\cG_0)$ such that $\c (B-\overline{B})\geq e_B$.
\end{enumerate}
\end{lem}
\begin{proof}
\begin{enumerate}
\item Since $|\cG_0|=L-\rank(M_{L-1})= \c (A)+1$, there exists an echelon form $E$ of $M_{2L}$ whose last $\c (A)+1$ rows contain the coefficients of $\partial^{\c (A)} A, \ldots ,\partial A, A$.  Then the last $\c (A)+1$ rows of the reduced echelon form $E_{2L}$ of $M_{2L}$ contain the coefficients of $B_{\c (A)}=\cD_{\c (A)} (A), \ldots, B_1=\cD_1( A), B_0=\cD_0 (A)$
for some $\cD_j\in\bbK [\partial]$ with $\deg (\cD_j)=j$ where $j=0,1,\ldots ,\c (A)$.
Then $\cG_0=\{B_0,B_1,\ldots ,B_{\c (A)}\}$.

\item  Let $s\in\{0,\ldots ,e_B-1\}$. Then by 1 the co-order of $B_{\c (A)-s}$ equals $\c (B_{\c (A)-s})=s$.
Since $B_{\c (A)-s}\in \bbK\{X\}$ then $\ord(B_{\c (A)-s},u_j)<N-\gamma_j-\gamma-s$, hence by Remark \ref{rem-kerSt} then $\l (B_{\c (A)-s})\in \ker (S^T)$.
Given $B\in \cG\backslash\cG_0$ we will prove by induction on $s$ that for $s=0,\ldots ,e_B-1$ there exists $C_s\in (\cG_0)$ such that $\c (B-C_s)\geq s+1$. Then $\overline{B}=C_{e_B-1}$.

There exists $\cG_i\in\bbK[\partial]$ with $\deg(\cG_i)\leq N-o_i-\gamma$, $i=1,\ldots ,n$ such that $B=\sum_{i=1}^n\cG_i (F_i(X,U))$. Let $\beta_i$ be the coefficient of $\partial^{N-o_i-\gamma}$ in $\cG_i$. By assumption $\ord(B,u_j)\leq N-\gamma_j-\gamma-e_B<N-\gamma_j-\gamma$ so $\beta=(\beta_1,\ldots ,\beta_n)\in\ker (S^T)$. Now $\rank (S)=n-1$ which means that $\dim \ker (S^T)=1$ so there exists $\mu\in\bbK$ such that $\beta=\mu l(B_{\c (A)})$. Let $C_0=\mu B_{\c (A)}$ then $\c (B-C_0)\geq 1$. This proves the claim for $s=0$.

Assuming the claim is true for $s-1$, $s\geq 1$ then there exists $C_{s-1}\in (\cG_0)$ such that $\c (B-C_{s-1})\geq s$. Then $B-C_{s-1}=\sum_{i=1}^n\cG_i^{s} (F_i(X,U))$ where $\cG_i^{s}\in\bbK[\partial]$ with $\deg(\cG_i^{s})\leq N-o_i-\gamma-s$, $i=1,\ldots ,n$. Let $\beta_i^s$ be the coefficient of $\partial^{N-o_i-\gamma-s}$ in $\cG_i^s$. By assumption $\ord(B-C_{s-1},u_j)\leq N-\gamma_j-\gamma-e_B<N-\gamma_j-\gamma-s$ so $\beta^s=(\beta_1^s,\ldots ,\beta_n^s)\in\ker (S^T)$. Now there exists $\mu_s\in\bbK$ such that $\beta^s=\mu_s l(B_{\c (A)-s})$. Let $C_s=C_{s-1}+\mu_s B_{\c (A)-s}$ then $\c (B-C_s)\geq s+1$.
\end{enumerate}
\end{proof}

\begin{thm}\label{thm-G0char}
Let $\cG$ be the Groebner basis associated to the system $\cP (X,U)$ with implicit ideal $\id$, $\cG_0=\cG\cap \bbK\{X\}$.  Let $S$ be the principal matrix of $\cP (X,U)$. If $\rank (S)=n-1$, for every nonzero linear $\id$-primitive differential polynomial $A$ in $(\cG_0)$ with $|\cG_0|= \c (A)+1$ then $A$ is a characteristic polynomial of $\id$.
\end{thm}
\begin{proof}
If $\c (A)=0$ then $\rank (M_{L-1})=L-1$. By Theorem 18(2) and Lemma 20(4) in \cite{RS} and Lemma \ref{lem-G0char}(1) then $A$ is a characteristic polynomial of $\id$.

Let us suppose that $\c (A)>0$. For $j\in\{1,\ldots ,n-1\}$ let
\[\cG_j=\{B\in\cG\mid \lead (B)=u_{jk}, k\in\{0,1,\ldots ,N-\gamma_j-\gamma\}\}.\]
Then $\cG\backslash\cG_0=\cup_{j=1}^{n-1}\cG_j$.
We use Algorithm \ref{alg-characteristic set} to prove the result. Set $\cA^0=\{B_0\}$ then by Lemma \ref{lem-G0char}(1) we have
$\prem (B_i,\cA^{i-1})=0$, $i=1,\ldots ,\c (A)$, thus $\cA^i=\cA^{0}$. Given $i\in\{\c (A)+1,\ldots ,L-1\}$ then $B_i\in\cG_{j_i}$ for some $j_i\in\{1,\ldots ,n-1\}$ and if $\cG_{j_i}\cap \cA^{i-1} =\emptyset$ then $\mbox{\prem}(B_i,\cA^{i-1})\notin\bbK\{X\}$ and $\cA^i=\cA^{i-1} \cup\{\mbox{\prem}(B_i,\cA^{i-1})\}$. If $i=\c (A)+1$ then $\cG_{j_i}\cap \cA^{i-1}=\emptyset$, we will prove by induction on $i$ that for $i=\c (A)+2,\ldots ,L-1$ if $\cG_{j_i}\cap \cA^{i-1} \neq\emptyset$ then $\prem (B_i,\cA^{i-1})=0$ and so $\cA^i=\cA^{i-1}$. This proves that $\cA_0=\cA^0$ and $A$ is a characteristic polynomial of $\id$.

For $i=\c (A)+2$ let us suppose that $\cG_{j_i}\cap \cA^{i-1}\neq\emptyset$ then $|\cG_0|=L-\rank (M_{L-1})= \c (A)+1$ implies
\[1\leq e_i=\ord (B_i,u_{j_i})-\ord (B_{i-1},u_{j_i})\leq \c (A).\]
Hence $\ord(B_{i-1},u_{j_i})\leq N-\gamma_{j_i}-\gamma-e_i$ and by Lemma \ref{lem-G0char}(2) there exists $\overline{B_i}\in (\cG_0)$ such that $\c (B_{i-1}-\overline{B_i})\geq e_i$. Thus $\partial^t (B_{i-1}-\overline{B_i})\in (\ps)$ for $t=1,\ldots ,e_i$. Therefore $C_i=\prem (B_i,B_{i-1}-\overline {B_i})\in (\ps)$ with $\lead (C_i)<\lead (B_{i-1})$ so $C_i\in (\ps)\cap\bbK\{X\}$. Since $\prem (B_{i-1}-\overline {B_i},\cA^{i-1})=0$ then $\prem (B_i,\cA^{i-1})=\prem (C_i,\cA^{i-1})= \prem (C_i,\cA^0)=0$.

Given $i\in \{\c (A)+3,\ldots ,L-1\}$ with $\cG_{j_i}\cap \cA^{i-1} \neq\emptyset$ then there exists $B\in \{B_{\c (A)+1},\ldots ,B_{i-1}\}$ such that
\[1\leq e_B=\ord (B_i,u_{j_i})-\ord (B,u_{j_i})\leq \c (A).\]
Hence $\ord(B,u_{j_i})\leq N-\gamma_{j_i}-\gamma-e_B$ and by Lemma \ref{lem-G0char}(2) there exists $\overline{B}\in (\cG_0)$ such that $\c (B-\overline{B})\geq e_B$.
Thus $\partial^{e_B} (B-\overline{B})\in (\ps)$. Therefore $C_i=B_i-\partial^{e_B} (B_{i-1}-\overline {B})\in (\ps)$ with $\lead (C_i)<\lead (B_{i})$ then $C_i=\gamma_0 B_0+\cdots +\gamma_{i-1} B_{i-1}$ with $\gamma_0,\ldots ,\gamma_{i-1}\in\bbK$. By induction on $i$ then
\[\prem (B_{i-1}-\overline {B},\cA^{i-1})=\prem(\prem(B_{i-1}-\overline{B},\prem(B_{i-1},\cA^{i-2})),\cA^{i-2})=0\]
so we have $\prem (B_i,\cA^{i-1})=\prem (C_i,\cA^{i-1})$. By induction we also have $\prem (B_0,\cA^{i-1})=0, \ldots , \prem (B_{i-1},\cA^{i-1})=0$ then
$\prem (C_i,\cA^{i-1})=0$.
\end{proof}

\begin{cor}\label{cor-dim}
Given a system $\cP (X,U)$ of linear DPPEs with implicit ideal $\id$.  Let $S$ and $M_{L-1}$ be the leading and principal matrices of $\cP(X,U)$ respectively. The following statements are equivalent.
\begin{enumerate}
\item The dimension of $\id$ is $n-1$.

\item $\rank (S)=n-1$ and there exists a nonzero linear $\id$-primitive differential polynomial $A$ such that $L-\rank(M_{L-1})=\c (A)+1$.

\item $\rank (S)=n-1$ and for every nonzero linear $\id$-primitive differential polynomial $B$ then $L-\rank(M_{L-1})=\c (B)+1$.
\end{enumerate}
In such situation $A(X)=0$ is the implicit equation of $\cP (X,U)$.
\end{cor}
\begin{proof}
By Theorem \ref{thm-Sdim} and Theorem \ref{thm-nec-coA} then (1)$\Rightarrow$(2) and by Corollary \ref{cor-nec-coA} (1)$\Rightarrow$(3).
Theorem \ref{thm-G0char} proves (3)$\Rightarrow$(1).
\end{proof}

%%%%%%%%%%%%%%%%%%%%%%%%%%%%%%%%%%%%%%%%%%%%%%%%%%%%%%%%%%%%%%%%%%%
\section{Linear perturbations of $\cP (X,U)$}\label{sec-Linear perturbation}
%%%%%%%%%%%%%%%%%%%%%%%%%%%%%%%%%%%%%%%%%%%%%%%%%%%%%%%%%%%%%%%%%%%

Let $\cP(X,U)$, $F_i$, $H_i$ be as in Section \ref{sec-Basic
notions and notation}.
Let $p$ be a differential indeterminate over $\bbK$ such that $\partial (p)=0$. Denote by
$\bbK_p=\bbK \langle p\rangle$ the differential field extension of
$\bbK$ by $p$. A {\sf linear perturbation of the system } $\cP (X,U)$ is a new system
\begin{equation*}\label{pDPPEs}
\cP_{\phi}(X,U) = \left\{\begin{array}{ccc}x_1 &= & P_1 (U)+p\,\phi_1(U)\\ & \vdots  & \\
x_n&= & P_n (U)+p\,\phi_n (U).\end{array}\right.
\end{equation*}
where the {\sf linear perturbation} $\phi=(\phi_1(U),\ldots ,\phi_n(U))$ is a family of linear differential polynomials in $\bbK\{U\}$.
For $i=1,\ldots ,n$ let
\[F^{\phi}_i(X,U)=F_i(X,U)-p\,\phi_i(U)\mbox{ and }H^{\phi}_i(U)=H_i(U)-p\,\phi_i(U).\]
Let $\ps_{\phi}=\ps (F^{\phi}_1,\ldots ,F^{\phi}_n)$, a set of linear differential polynomials in $\bbK_p [\cX][\cV]\subset\bbK_p \{X\cup U\}$ and let $(\ps_{\phi})$ be the ideal generated by $\ps_{\phi}$ in $\bbK_p [\cX][\cV]$.
We prove next the existence of a linear perturbation $\phi$ such that $\dcres (F_1^{\phi},\ldots ,F_n^{\phi})\neq 0$.

\para

Let us suppose that $o_n\geq o_{n-1}\geq \ldots \geq o_1$ to define the perturbation $\phi=(\phi_1(U),\ldots ,\phi_n(U))$ by
\begin{equation}\label{eq-phi}
\phi_i (U)=\left\{\begin{array}{ll}\varepsilon_i u_{n-i-1, o_i-\gamma_{n-i-1}}+ u_{n-i},& i=1,\ldots ,n-2,\\
u_1,& i=n-1,\\
\varepsilon_n u_{n-1, o_n-\gamma_{n-1}},& i=n,
\end{array}\right.
\end{equation}
where $\varepsilon_i=1$ if $o_i\neq 0$ and  $\varepsilon_i=0$ if $o_i= 0$, for $i=1,\ldots ,n$.
\para

Let us suppose that $N\geq 1$. We denote by $M_{\phi}(L^h)$ the complete differential homogeneous resultant matrix for the set of
linear differential polynomials $H^{\phi}_1,\ldots ,H^{\phi}_n$. Then  $M_{\phi} (L^h)$ is an $L^h\times L^h$ matrix with elements in $\bbK[p]$ and there exists an $L^h\times L^h$ matrix $M_{\phi}$ with elements in $\bbK$ such that $M_{\phi} (L^h)=M(L^h)-p\,M_{\phi}$. Then
\[\dcres^h (H^{\phi}_1,\ldots ,H^{\phi}_n)=\det (M_{\phi}(L^h))=\det(M(L^h)-p\,M_{\phi}).\]
As expected, the linear perturbation that makes  $\dcres (F_1^{\phi},\ldots ,F_n^{\phi})\neq 0$ is not unique. Our proposal can be viewed as a generalization of
 the characteristic polynomial of a matrix to the linear differential case in the spirit of \cite{Can}. Namely, for
$o_n=1$, $o_{n-1}=\cdots =o_1=0$ then we obtain the characteristic polynomial of $M(L^h)$, namely $\dcres^h (H_1-p\,\phi_1,\ldots ,H_n-p\, \phi_n)=\det(M(L^h)-p\, I_{L^h})$ where $I_{L^h}$ is the $L^h\times L^h$ identity matrix.

Let $S^{\phi}$ be the leading matrix of $\cP_{\phi}(X,U)$. For $i\in\{1,\ldots n\}$ let $S_i^{\phi}$ be
the $(n-1)\times (n-1)$ matrix obtained by removing the $i$-th row of $S^{\phi}$.

\begin{prop}\label{prop-dResh-nonzero}
Given a system $\cP (X, U)$ of linear DPPEs and the perturbation $\phi$ defined by \eqref{eq-phi}.
The following statements hold.
\begin{enumerate}
\item The determinant of $S^{\phi}_n$ is nonzero and it has degree $n-1$ in $p$.

\item If $N\geq 1$ then the linear complete homogeneous differential resultant
$\dcres^h (H^{\phi}_1,\ldots ,H^{\phi}_n)$ is a polynomial in $\bbK [p]$ of degree $L^h$ and not identically zero.
\end{enumerate}
\end{prop}
\begin{proof}
\begin{enumerate}
\item Observe that $S^{\phi}_n$ has $p$'s in the main diagonal.

\item We have $\det (M_{\phi}(L^h))=p^{L^h}\det((1/p)M(L^h)-M_{\phi})$. If we set $y=1/p$ the matrix obtained from $y M(L^h)-\,M_{\phi}$ at $y=0$ is $M_{\phi}$. We will prove that $\det (M_{\phi})\neq 0$ and therefore the degree of  $\det (M_{\phi}(L^h))$ in $p$ is $L^h$.

In each row and column of $M_{\phi}$ there is at least one nonzero entry and at most two. Also, if a column has two nonzero entries at least one of them is the only nonzero entry of its row. Therefore we can reorganize the rows
of $M_{\phi}$ to get a matrix $N$ which has ones in the main diagonal and in every row at most one nonzero entry not in the main diagonal. Namely the $i$-th row of $N$ is either the only row of $M_{\phi}$ with a nonzero entry in the $i$-th column or the row of $M_{\phi}$ with its only nonzero entry in the $i$-th column. Thus $N$ is a product
of elementary matrices and so it has nonzero determinant.
\end{enumerate}
\end{proof}

\begin{thm}
Given a system $\cP (X, U)$ of linear DPPEs there exists a linear perturbation $\phi$ such that
the differential resultant
$\dcres (F_1^{\phi},\ldots ,F_n^{\phi})$ is a nonzero polynomial in $\bbK [p]\{X\}$ and $\det (S^{\phi}_n)\neq 0$.
\end{thm}
\begin{proof}
Let $\phi$ be the perturbation defined by \eqref{eq-phi} then $\det (S^{\phi}_n)\neq 0$ by Proposition \ref{prop-dResh-nonzero}.
If $N=0$ the result follows from
\[\dcres (F_1^{\phi},\ldots ,F_n^{\phi})=\sum_{i=1}^n (-1)^{i+n} \det(S_i^\phi) (x_i-a_i).\]
If $N\geq 1$ then $\dcres^h (H^{\phi}_1,\ldots ,H^{\phi}_n)\neq 0$ by Proposition \ref{prop-dResh-nonzero}. This is equivalent by \cite{RS}, Theorem 18(2) to $\dcres (F_1^{\phi},\ldots ,F_n^{\phi})\neq 0$.
\end{proof}

If nonzero then $\dcres (F_1^{\phi},\ldots ,F_n^{\phi})$ is a polynomial in $p$ whose coefficients are linear differential polynomials in $\bbK \{X\}$. We focus our attention next in the coefficient of the lowest degree term in $p$ of $\dcres (F_1^{\phi},\ldots ,F_n^{\phi})$.

\begin{thm}
Given a system $\cP (X, U)$ of linear DPPEs let $\phi$ be a linear perturbation  such that $\dcres (F_1^{\phi},\ldots ,F_n^{\phi})\neq 0$ and $\det (S^{\phi}_n)\neq 0$. The following statements hold.
\begin{enumerate}
\item There exists a linear differential polynomial $P$ in $(\ps_{\phi})\cap \bbK_p\{X\}$ with coefficients in $\bbK[p]$ and content in $\bbK$ such that $\alpha_n\neq 0$ where $l(P)=(\alpha_1,\ldots ,\alpha_n)$.

\item  There exists $a\in\bbN$ such that
\begin{equation}\label{eq-formuladRes}
\alpha_{n} \dcres (F_1^{\phi},\ldots ,F_n^{\phi})=(-1)^a \det(S^{\phi}_n) \dcres^h (H^{\phi}_1,\ldots ,H^{\phi}_n) P(X).
\end{equation}
Furthermore $\frac{\det(S_n^{\phi})\dcres^h(H^{\phi}_1,\ldots ,H^{\phi}_n)}{\alpha_n}\in\bbK[p]$.
\end{enumerate}
\end{thm}
\begin{proof}
\begin{enumerate}
\item By Lemma \ref{lem-G0} there exists a nonzero linear differential polynomial $B\in (\ps_{\phi})\cap \bbK_p\{X\}$. There exists $c\in\bbK_p$ such that $P(X)=c \partial^{\c (B)} B(X)$ has coefficients in $\bbK[p]$ and content in $\bbK$. Observe that $\co (P)=0$.
    Since $P\in \bbK_p\{X\}$ then $\ord(P,u_j)<N-\gamma_j-\gamma$ and by Remark \ref{rem-kerSt}
it holds $l(B)\in\ker ((S^{\phi})^T)$. If $\alpha_n =0$ then $\det (S^{\phi}_n)\neq 0$ implies
$\alpha_i=0$, $i=1,\ldots ,n$. This contradicts that $\c (P)=0$ therefore $\alpha_n \neq 0$.

\item Equation \eqref{eq-formuladRes} follows from \cite{RS}, Theorem 18(1). Since the content of $P(X)$ belongs to $\bbK$ then $\alpha_n$ divides $\det(S_n^{\phi})\dcres^h(H^{\phi}_1,\ldots ,H^{\phi}_n)$ in $\bbK[p]$.
\end{enumerate}
\end{proof}

\begin{rem} In the situation of the previous theorem and with the perturbation $\phi$ defined by \eqref{eq-phi} we can add the following remarks.
For $i=1,\ldots ,n$ by formula \eqref{eq-Mxi}  then $\det(S^{\phi}_i) \dcres^h (H^{\phi}_1,\ldots ,H^{\phi}_n)$ is the coefficient of $x_{i N-o_i-\gamma}$ in $\dcres (F_1^{\phi},\ldots ,F_n^{\phi})$.
By 1 and 2 in the previous theorem then
\[\alpha_n \det(S^{\phi}_i) \dcres^h (H^{\phi}_1,\ldots ,H^{\phi}_n) = \det(S^{\phi}_n) \dcres^h (H^{\phi}_1,\ldots ,H^{\phi}_n)\alpha_i,\]
hence  $\det(S_i^{\phi})=\det(S_n^{\phi})\alpha_i/\alpha_n\in \bbK[p]$.
Let $\alpha=\gcd(\alpha_n, \det(S_n^{\phi}))$. If $\alpha\in\bbK$ then
the degree of $\det(S_i^{\phi})$ in $p$ is greater or equal than $n-2$ which is not possible
then $\gcd(\alpha_n, \det(S_n^{\phi}))\in\bbK[p]\backslash\bbK$.
\end{rem}

Let $D_{\phi}$ be the lowest degree of $p$ in $\dcres (F_1^{\phi},\ldots ,F_n^{\phi})$ and let
$A_{D_{\phi}}$ be the coefficient of $p^{D_{\phi}}$ in $\dcres (F_1^{\phi},\ldots ,F_n^{\phi})$.
 We call $D_{\phi}$ the {\sf degree of the perturbed system} $\cP_{\phi} (X,U)$. Observe that
\[D_{\phi}=0\Leftrightarrow\dcres (F_1,\ldots ,F_n)\neq 0.\]
We write $D_{\phi}=-1$ if $\dcres (F_1^{\phi},\ldots ,F_n^{\phi})= 0$ and so
\[D_{\phi}\geq 0\Leftrightarrow \dcres (F_1^{\phi},\ldots ,F_n^{\phi})\neq 0.\]

In the remaining parts of this section we assume that $\dcres (F_1^{\phi},\ldots ,F_n^{\phi})\neq 0$.
Let $\ps=\ps(F_1,\ldots,F_n)$ and let $\id$ be the implicit ideal of $\cP(X,U)$.
We will use
$\dcres (F_1^{\phi},\ldots ,F_n^{\phi})$ to provide a nonzero $\id$-primitive differential polynomial $A_{\phi}$ in $(\ps)\cap\bbK\{X\}$.

\para

\begin{lem}\label{lem_AD}
The linear differential polynomial $A_{D_{\phi}}$ belongs to $(\ps)\cap\bbK\{X\}$.
\end{lem}
\begin{proof}
By \cite{RS}, Proposition 16 then $\dcres (F_1^{\phi},\ldots ,F_n^{\phi})\in (\ps_{\phi})\cap\bbK_p\{X\}$
and it equals $p^{D_{\phi}} (A_{D_{\phi}}+p A')$ for some $A'\in\bbK_p\{X\}$. Therefore the linear polynomial
$A_{D_{\phi}}+p A'\in (\ps_p)\cap\bbK_p\{X\}$ and there exist $\cF_i\in\bbK_p[\partial]$ with $\deg(\cF_i)\leq N-o_i-\gamma$
such that $A_{D_{\phi}}(X)+p A'(X)=\sum_{i=1}^n \cF_i(F_i^{\phi}(X,U))$. Then $A_{D_{\phi}}(X)+p A'(X)=\sum_{i=1}^n \cF_i(x_i-a_i)$
and $\sum_{i=1}^n \cF_i(H_i^{\phi}(U))=0$.

For each $i\in\{1,\ldots ,n\}$ there exists $\cL_i\in\bbK[\partial]$ and $\cF'_i\in\bbK_p[\partial]$ such that $\cF_i=\cL_i+p\cF'_i$. Then $A_{D_{\phi}}(X)+p A'(X)=\sum_{i=1}^n \cL_i(x_i-a_i)+p\sum_{i=1}^n \cF'_i(x_i-a_i)$ and hence $A_{D_{\phi}}(X)=\sum_{i=1}^n \cL_i(x_i-a_i)$. On the other hand
\[0=\sum_{i=1}^n \cF_i(H_i^{\phi}(U))= \sum_{i=1}^n \cL_i(H_i(U))+ p\sum_{i=1}^n \cL_i(\phi_i(U))+p\sum_{i=1}^n \cF'_i(H_i^{\phi}(U))\]
which implies $\sum_{i=1}^n \cL_i(H_i(U))=0$. Thus
$A_{D_{\phi}}(X)=\sum_{i=1}^n \cL_i(F_i(X,U))$
with $\deg(\cL_i)\leq N-o_i-\gamma$ then $A_{D_{\phi}}\in (\ps)\cap\bbK\{X\}$.
\end{proof}

Let $A_{\phi}$ be the $\id$-primitive part of $A_{D_{\phi}}$. We call $A_{\phi}$ the {\sf differential polynomial associated to} $\cP_{\phi} (X,U)$.
We relate $D_{\phi}$ with $\c (A_{\phi})$ and give conditions on $D_{\phi}$ for $A_{\phi}(X)=0$ to be the implicit equation of $\cP (X,U)$.

\begin{rem} Let $\cP_{\phi} (X,U)$ and $\cP_{\psi} (X,U)$ be two different linear perturbations of $\cP(X,U)$ with degrees $D_{\phi}\geq 0$ and  $D_{\psi}\geq 0$. Let $A_{\phi}$ and $A_{\psi}$ be the associated differential polynomials.
\begin{enumerate}
\item As illustrated in Example 1 of Section \ref{sec-Implicitization algorithm} the degrees $D_{\phi}$ and $D_{\psi}$ may be different.

\item If $\dim\id=n-1$ then $A_{\phi}=\gamma A_{\psi}$ for some $\gamma\in\bbK$.
\end{enumerate}
\end{rem}

\begin{thm}\label{thm-D}
Let $\cP_{\phi} (X,U)$ be a perturbed system of the system $\cP (X,U)$ of degree $D_{\phi}\geq 0$. Let $\cG$ be the Groebner basis associated to $\cP (X,U)$ and $\cG_0=\cG\cap \bbK\{X\}$. Then $|\cG_0|-1\leq D_{\phi}$.
\end{thm}
\begin{proof}
Let $M_{\phi}(L)$ be the differential resultant matrix of $F_1^{\phi}, \ldots ,F_n^{\phi}$ and $M^{\phi}_{L-1}$ the $L\times (L-1)$ principal submatrix of $M_{\phi}(L)$.
 By equation \eqref{eq-dResMx} then
 \[\dcres (F_1^{\phi},\ldots ,F_n^{\phi})=\sum_{i=1}^n \sum_{k=0}^{N-o_i-\gamma} b_{ik} \det (M^{\phi}_{x_{ik}}) (x_{ik}-\partial^k a_i),\] 
 with $M^{\phi}_{x_{ik}}$ an $(L-1)\times (L-1)$ submatrix of $M_{L-1}^{\phi}$ and $b_{ik}=\pm 1$ according to the row index of $x_{ik}-\partial^k a_i$ in the matrix $M_{\phi}(L)$. Let $M_{x_{ik}}$ be the $(L-1)\times (L-1)$ submatrix of the principal matrix $M_{L-1}$ of $\cP (X,U)$.
 Let $r_{ik}=\rank(M_{x_{ik}})$, then there exists an invertible matrix $E_{ik}$ of order $L-1$ and entries in $\bbK$ such that the last $L-1-r_{ik}$ rows of $E_{ik} M_{x_{ik}}$ are zero. If we divide each one of the last $L-1-r_{ik}$ rows of $E_{ik} M^{\phi}_{x_{ik}}$ by $p$ we obtain a matrix $N^{\phi}_{ik}$ such that
 \[\det M_{\phi}(L)=\sum_{i=1}^n \sum_{k=0}^{N-o_i-\gamma} b_{ik} p^{L-1-r_{ik}}\det (N^{\phi}_{ik}) (x_{ik}-\partial^k a_i).\]
 This proves that $L-1-r_{ik}\leq D_{\phi}$ for all $i=1,\ldots ,n$ and $k=0,1,\ldots ,N-o_i-\gamma$.
 Now $\rank (M_{L-1})\geq r_{ik}$ so \[|\cG_0|-1=L-\rank (M_{L-1})-1\leq L-1-r_{ik}\leq D_{\phi}.\]
\end{proof}

We showed that $D_{\phi}\geq |\cG_0|-1=L-\rank (M_{L-1})-1\geq\c (A_{\phi})$ and
in general equality does not hold ( see examples in Section \ref{sec-Implicitization algorithm}).

\begin{cor}
Let $\cP (X,U)$ be a system of linear DPPEs with implicit ideal $\id$ and leading matrix $S$.
Let $\cP_{\phi} (X,U)$ be a perturbed system of $\cP (X,U)$ of degree $D_{\phi}\geq 0$. Let $A_{\phi}$ be the differential polynomial associated to $\cP_{\phi} (X,U)$. If $\rank (S)=n-1$ and $D_{\phi}=\c (A_{\phi})$ then $\id$ has dimension $n-1$ and $A_{\phi}(X)=0$ is the implicit equation of $\cP (X,U)$.
\end{cor}
\begin{proof}
If $D_{\phi}=\c (A_{\phi})$ then by Theorem \ref{thm-D} we have $|\cG_0|\leq \c (A_{\phi})+1$. By Lemma \ref{lem-G0} then $|\cG_0|= \c (A_{\phi})+1$ and by Corollary \ref{cor-dim} the result follows.
\end{proof}

%%%%%%%%%%%%%%%%%%%%%%%%%%%%%%%%%%%%%%%%%%%%%%%%%%%%%%%%%%%%%%%%%%%
\section{Implicitization algorithm for linear DPPEs and examples}\label{sec-Implicitization algorithm}
%%%%%%%%%%%%%%%%%%%%%%%%%%%%%%%%%%%%%%%%%%%%%%%%%%%%%%%%%%%%%%%%%%%

Let $\cP (X,U)$ be a system of linear DPPEs with implicit ideal $\id$.
In this section, we give an algorithm that decides whether the dimension of $\id$ is $n-1$ and in the affirmative case returns the implicit equation of $\cP (X,U)$. For this purpose let $S$ and $M_{L-1}$ be the leading and principal matrices of $\cP (X,U)$ respectively.  Let $\cP_{\phi} (X,U)$ be a perturbed system of $\cP (X,U)$ of degree $D_{\phi}\geq 0$. Let $A_{D_{\phi}}$ be the coefficient of $p^{D_{\phi}}$ in $\dcres (F_1^{\phi},\ldots ,F_n^{\phi})$ and $A_{\phi}$ the differential polynomial associated to $\cP_{\phi} (X,U)$.

\begin{alg}
\begin{itemize}
\item \underline{\sf Given}  the system $\cP (X,U)$ of linear DPPEs.
\item \underline{\sf Decide} whether the dimension is $n-1$ and in the affirmative case \item
\underline{\sf Return} a characteristic polynomial of $\id$.
\end{itemize}
\begin{enumerate}
\item Compute $\rank (S)$.

\item If $\rank (S)<n-1$ {\sf RETURN} ``{\sf dimension less than $n-1$}".

\item Compute $\cP_{\phi} (X,U)$ with perturbation $\phi$ given by \eqref{eq-phi}.

\item Compute $\dcres (F_1^{\phi},\ldots ,F_n^{\phi})$, $D_{\phi}$ and $A_{D_{\phi}}$.

\item If $D_{\phi}=0$ {\sf RETURN} $A_{D_{\phi}}$.

\item Compute $A_{\phi}$ and $\c (A_{\phi})$.

\item If $D_{\phi}=\c (A_{\phi})$ {\sf RETURN} $A_{\phi}$.

\item Compute $\rank (M_{L-1})$.

\item If $L-\rank (M_{L-1})>\c (A_{\phi})+1$ {\sf RETURN} ``{\sf dimension less than $n-1$}".

\item If $L-\rank (M_{L-1})=\c (A_{\phi})+1$ {\sf RETURN} $A_{\phi}$.
\end{enumerate}
\end{alg}

The next examples were computed in Maple. The computation of differential resultants was carried out with our Maple implementation
of the linear complete differential resultant, available at \cite{R}.

%%%%%%%%%%%%%%%%%%%%%%%%%%%%%%%%%%%%%%%%%%%%%%%%%%%%%%%%%%%%%%%%%%%
\subsection{Example 1}\label{subsec-ex1}
%%%%%%%%%%%%%%%%%%%%%%%%%%%%%%%%%%%%%%%%%%%%%%%%%%%%%%%%%%%%%%%%%%%

Let $\bbK=\bbQ$, $\partial =\frac{\partial}{\partial t}$ and
consider the system $\cP (X,U)$ of linear DPPEs providing the set of differential polynomials in
$\bbK\{x_1,x_2,x_3\}\{u_1,u_2\}$,
\begin{align*}
F_1 (X,U)&=x_1+u_1-u_2+u_{11}-u_{12}-4u_{21}-3u_{22},\\
F_2 (X,U)&=x_2+u_2+u_{11}-u_{22},\\
F_3 (X,U)&= x_3+u_2+u_{11}+u_{21}.
\end{align*}
The set $\ps(F_1,F_2,F_3)$ contains $L=13$ differential
polynomials and $\gamma=0$.
The leading matrix $S$ of $\cP(X,U)$ has rank $2$ and equals
\[
S=\left[
{\begin{array}{rr}
-3 & 1 \\
-1 & 0 \\
1 & 1
\end{array}}
 \right].
\]
We consider the perturbation $\phi=(\phi_1(U),\phi_2(U),\phi_3(U))$ with
\begin{equation*}
\phi_i (U)=\left\{\begin{array}{ll}
u_{12}+u_2,& i=1,\\
u_1,& i=2,\\
u_{21},& i=3.
\end{array}\right.
\end{equation*}

There exists a differential polynomial $P(X)\in (\ps)\cap\bbK\{X\}$ with coefficients in $\bbK [p]$ and content in $\bbK$ such that the determinant of the $13\times 13$ matrix $M_{\phi} (13)$ equals
\begin{align*}
&\dcres (F_1^{\phi},F_2^{\phi},F_3^{\phi})=\dcres^h (H_1^{\phi},H_2^{\phi},H_3^{\phi}) P(X)=\\
&p\,(1 + 4\,p + 4\,p^{4} - p^{5} + 2\,p^{7} + 11\,p^{3} + p^{9} - 12
\,p^{2} - 4\,p^{6} + p^{8}) P(X).
\end{align*}
Then $D_{\phi}=1$ and the coefficient of $p$ in $\dcres (F_1^{\phi},F_2^{\phi},F_3^{\phi})$ is
\begin{align*}
&A_{D_{\phi}}={x_{12}} - {x_{2}} - 2\,{x_{21}} - 2\,{x_{22}} + {x_{33}} +
{x_{32}} + {x_{31}} + {x_{3}}.
\end{align*}
We have $A_{D_\phi}=\cL_1(x_1)+\cL_2(x_2)+\cL_3(x_3)$ with
\begin{align*}
\cL_1&={\partial}^{2} + {\partial}^{3}={\partial}^{2}\,(1 + {\partial}),\\
\cL_2&= - 1 - 3\,{\partial} - 4\,{\partial}^{
2} - 2\,{\partial}^{3}= - ({\partial} + 1)\,(2\,{\partial}^{2} + 2\,{\partial} + 1),\\
\cL_3&= 1 + 2\,{\partial} + 2\,{\partial}^{2}
 + 2\,{\partial}^{3} + {\partial}^{4}=({\partial}^{2} + 1)\,({\partial} + 1)^{2}.
\end{align*}
Therefore $\cL=\gcld (\cL_1,\cL_2,\cL_3)=1+\partial$ and $A_{\phi}={x_{12}} - {x_{2}} - 2\,{x_{21}} - 2\,{x_{22}} + {x_{33}} +
{x_{32}} + {x_{31}} + {x_{3}}$ with
$\c (A_{\phi})=1$. Then $D_{\phi}=\c (A_{\phi})$.
We conclude that the dimension of $\id$
is $n-1=2$ and its implicit equation $A_{\phi} (X)=0$.

If we consider the perturbation $\psi=(\psi_1(U),\psi_2(U),\psi_3(U))$ with
\begin{equation*}
\psi_i (U)=\left\{\begin{array}{ll}
u_{22}+u_1,& i=1,\\
u_2,& i=2,\\
u_{11},& i=3.
\end{array}\right.
\end{equation*}
There exists a differential polynomial $P(X)\in (\ps)\cap\bbK\{X\}$ with coefficients in $\bbK [p]$ and content in $\bbK$ such that the determinant of the $13\times 13$ matrix $M_{\psi} (13)$ equals
\begin{align*}
&\dcres (F_1^{\psi},F_2^{\psi},F_3^{\psi})=\dcres^h (H_1^{\psi},H_2^{\psi},H_3^{\psi}) P(X)=\\
&- p^{2}\,(1 + p)\,(p^{7} + p^{6} - 17\,p^{5} - 53\,p^{4} - 60\,p
^{3} - 33\,p^{2} - 9\,p - 1) P(X).
\end{align*}
Then $D_{\psi}=2$ but the coefficient of $p^2$ is $A_{D_{\psi}}=A_{D_{\phi}}$. Therefore $A_{\psi}=A_{\phi}$ as expected.

%%%%%%%%%%%%%%%%%%%%%%%%%%%%%%%%%%%%%%%%%%%%%%%%%%%%%%%%%%%%%%%%%%%
\subsection{Example 2}\label{subsec-ex1}
%%%%%%%%%%%%%%%%%%%%%%%%%%%%%%%%%%%%%%%%%%%%%%%%%%%%%%%%%%%%%%%%%%%

Let $\bbK=\bbQ$, $\partial =\frac{\partial}{\partial t}$ and
consider the system $\cP (X,U)$ of linear DPPEs providing the set of differential polynomials in
$\bbK\{x_1,x_2,x_3,x_4\}\{u_1,u_2,u_3\}$,
\begin{align*}
F_1 (X,U)&= x_1-2u_1+u_3-3u_{21}+u_{31},\\
F_2 (X,U)&= x_2+2u_1-u_3-u_{11}+3u_{22}-u_{32},\\
F_3 (X,U)&= x_3+2u_1-u_3+2u_{21}+u_{32},\\
F_4 (X,U)&= x_4+2u_1-u_3+3u_{21}-2u_{31}.
\end{align*}
The set $\ps(F_1,F_2,F_3,F_4)$ contains $L=18$ differential
polynomials and $\gamma=\gamma_1=1$.
The leading matrix $S$ of $\cP(X,U)$ has rank $3$ and equals
\[
S=\left[
{\begin{array}{rrr}
1 & -3 & -2 \\
-1 & 3 & -1 \\
1 & 0 & 0 \\
-2 & 3 & 2
\end{array}}
 \right].
\]
We consider the perturbation $\phi=(\phi_1(U),\phi_2(U),\phi_3(U))$ with
\begin{equation*}
\phi_i (U)=\left\{\begin{array}{ll}
u_{21}+ u_3,& i=1,\\
u_{11}+ u_2,& i=2,\\
u_1,& i=3,\\
u_{31},& i=4.
\end{array}\right.
\end{equation*}
There exists a differential polynomial $P(X)\in (\ps)\cap\bbK\{X\}$ with coefficients in $\bbK [p]$ and content in $\bbK$ such that the determinant of the $18\times 18$ matrix $M_{\phi} (18)$ equals
\begin{align*}
&\dcres (F_1^{\phi},F_2^{\phi},F_3^{\phi} ,F_4^{\phi})=\dcres^h (H_1^{\phi},H_2^{\phi},H_3^{\phi} ,H_4^{\phi}) P(X)=\\
&p^3\,(-2p^9-3944p^4+789p^2-481p^7 -379p^6+108+4484p^5+212p^8+p^{11}\\
&-642p+527p^3-7p^{10}) P(X).
\end{align*}
Then $D_{\phi}=3$ and the coefficient of $p^3$ in $\dcres (F_1^{\phi},F_2^{\phi},F_3^{\phi} ,F_4^{\phi})$ is
\begin{align*}
&A_{D_{\phi}}=\\
&972x_{12}-864x_{13}-972x_{14}-216x_{22}+648x_{32}-972x_{33}+540x_{42}-972x_{44}.
\end{align*}
We have $A_D=\cL_1(x_1)+\cL_2(x_2)+\cL_3(x_3)+\cL_4(x_4)$ with
\begin{align*}
\cL_1&=-864\partial^3+972\partial^2-972\partial^4=-108\partial^2(8\partial-9+9\partial^2), \\
\cL_2&= -216\partial^2,\\
\cL_3&= -972\partial^3+648\partial^2=-324\partial^2(3\partial-2),\\
\cL_4&= 540\partial^2-972\partial^4=-108\partial^2(-5+9\partial^2).
\end{align*}
Therefore $\cL=-108\partial^2$ and $A_{\phi}=8x_{11}+9x_{12}-9x_{1}+2x_{2}-6x_{3}+9x_{31}+9x_{42}-5x_{4}$ with
$\co (A_{\phi})=2$. Then $D_{\phi}>\c (A_{\phi})$.

Replace $p$ in $M_{\phi}(18)$ by zero to obtain $M(18)$ whose principal $18\times 17$ submatrix is $M_{L-1}$.
Compute $L-\rank (M_{L-1})=3$. Then $\c (A_{\phi})+1=L-\rank (M_{L-1})$. We conclude that the dimension of $\id$
is $n-1=3$ and its implicit equation $A_{\phi} (X)=0$.

%%%%%%%%%%%%%%%%%%%%%%%%%%%%%%%%%%%%%%%%%%%%%%%%%%%%%%%%%%%%%%%%%%%
\subsection{Example 3}\label{subsec-ex1}
%%%%%%%%%%%%%%%%%%%%%%%%%%%%%%%%%%%%%%%%%%%%%%%%%%%%%%%%%%%%%%%%%%%

Let $\bbK=\bbQ(t)$, $\partial =\frac{\partial}{\partial t}$ and
consider the system $\cP (X,U)$ of linear DPPEs providing the set of differential polynomials in
$\bbK\{x_1,x_2,x_3\}\{u_1,u_2\}$,
\begin{align*}
{F_{1}} &= {x_1} - 3 + {{u_{11}}} + {{u_{12}}}- {u_2} - 4\,{{u_{21}}}
 - 3\,{{u_{22}}}, \\
{F_{2}} &= {x_2} + {{u_{11}}}+ {u_2} - {{u_{22}}}, \\
{F_{3}} &= {x_3} + 2 + {{u_{11}}}+ t\,{u_2} + {{u_{21}}}.
\end{align*}
Then the set $\ps(F_1,F_2,F_3)$ contains $L=13$ differential
polynomials and $\gamma=0$.
The leading matrix $S$ of $\cP(X,U)$ has rank $2$ and equals
\[
S=\left[
{\begin{array}{rr}
-3 & 1 \\
-1 & 0 \\
1 & 1
\end{array}}
 \right].
\]
We consider the perturbation $\phi=(\phi_1(U),\phi_2(U),\phi_3(U))$ with
\begin{equation*}
\phi_i (U)=\left\{\begin{array}{ll}
u_{12}+u_2,& i=1,\\
u_1,& i=2,\\
u_{21},& i=3.
\end{array}\right.
\end{equation*}
There exists a differential polynomial $P(X)\in (\ps)\cap\bbK\{X\}$ with coefficients in $\bbK [p]$ and content in $\bbK$ such that the determinant of the $13\times 13$ matrix $M_{\phi} (13)$ equals $\dcres (F_1^{\phi},F_2^{\phi},F_3^{\phi})=p\,P(X)$. In this case $\alpha_3\,p=\det(S_3^{\phi})\dcres^h(H^{\phi}_1,H^{\phi}_2 ,H^{\phi}_3)$ where $l(P)=(\alpha_1,\alpha_2,\alpha_3)$ is the leading vector of $P$.
Then $D_{\phi}=1$ and the coefficient $A_{D_{\phi}}$ of $p$ in $\dcres (F_1^{\phi},F_2^{\phi},F_3^{\phi})$ equals
\[A_{D_{\phi}}=\cL_1(x_1-3)+\cL_2(x_2)+\cL_3(x_3+2)\] with
\begin{align*}
\cL_1&= - 64\,t^{2} - 2656 + 912\,t - 8\,t^{3} + (468 +
57\,t^{3} + 33\,t^{2} - t^{5} - 13\,t^{4} - 522\,t)\,{\partial} \\
&\mbox{} + ( - 364 + 79\,t^{2} - 6\,t - 10\,t^{3} - t^{4})\,{\partial}^{2} \\
&\mbox{} + (68\,t + 37\,t^{2} - 14\,t^{3} - t^{4} - 296)\,{\partial}^{3},\\
\cL_2&= (730\,t + 860 - 37\,t^{3} - 229\,t^{2} + 15\,t^{4}
 + t^{5})\,\partial\\
&\mbox{} + (252\,t + 19\,t^{4} + t^{5} + 1752 - 421\,t^{2} + 17\,t
^{3})\,\partial^{2}\\
&\mbox{} + (56\,t^{3} - 272\,t - 148\,t^{2} + 4\,t^{4} + 1184)\,
\partial^{3},\\
\cL_3&= - 16\,t^{2} + 228\,t - 2\,t^{3} - 664 + ( - 2\,t
^{4} + 212\,t^{2} - 436\,t - 18\,t^{3} - 664)\,{\partial}\\
&\mbox{} + ( - 1856 - 64\,t^{3} + 309\,t^{2} - 5\,t^{4} + 276\,t)\,{\partial}^{2} \\
&\mbox{} + ( - 2\,t^{4} + 210\,t + 32\,t^{2} - 524 - 32\,t^{3})\,{\partial}^{3} \\
&\mbox{} + ( - 68\,t - 37\,t^{2} + 14\,t^{3} + t^{4} + 296)\,{\partial}^{4}.
\end{align*}
Using the Maple package OreTools we check that $(\cL_1,\cL_2,\cL_3)=1$, then $A_{D_{\phi}}=A_{\phi}$ and $\co (A_{\phi})=0$.
Replace $p$ in $M_{\phi}(13)$ to obtain $M(13)$ whose principal $13\times 12$ submatrix is $M_{L-1}$.
Compute $L-\rank (M_{L-1})=2$ which is greater than $\c (A_{\phi})+1=1$. We conclude that the dimension of $\id$
smaller than $n-1=2$.

If we take a different perturbation $\psi=(\psi_1(U),\psi_2(U),\psi_3(U))$ with
\begin{equation*}
\psi_i (U)=\left\{\begin{array}{ll}
u_{22}+u_1,& i=1,\\
u_2,& i=2,\\
u_{11},& i=3.
\end{array}\right.
\end{equation*}
We obtain $D_{\psi}=1$ and the coefficient of $p$ in $\dcres (F_1^{\psi},F_2^{\phi},F_3^{\psi})$ is
\[A_{D_{\psi}}=\cK_1(x_1-3)+\cK_2(x_2)+\cK_3(x_3+2)\] with
\begin{align*}
\cK_1&= ( - 76\,t^{3} + 154\,t^{2} + 12\,t - 156 + 11\,t
^{4} + t^{5})\,\partial \\
&\mbox{} + (204\,t - 100 + t^{4} + 8\,t^{3} - 94\,t^{2})\,
\partial^{2} \\
&\mbox{} + ( - 58\,t^{2} + 12\,t^{3} + 88\,t - 12 + t^{4})\,
\partial^{3} ,\\
\cK_2&= - 228\,t + 16\,t^{2} + 664 + 2\,t^{3} + ( - t^{5
} - 904\,t + 820 + 62\,t^{3} + 90\,t^{2} - 13\,t^{4})\,{\partial} \\
&\mbox{} + ( - t^{5} - 17\,t^{4} + 346\,t^{2} - 892\,t + 412 + 14
\,t^{3})\,{\partial}^{2} \\
&\mbox{} + (232\,t^{2} - 352\,t - 4\,t^{4} - 48\,t^{3} + 48)\,
{\partial}^{3},\\
\cK_3&=  - 228\,t + 16\,t^{2} + 664 + 2\,t^{3} + (14\,t^{
3} + 892\,t - 664 + 2\,t^{4} - 244\,t^{2})\,{\partial} \\
&\mbox{} + ( - 156 - 406\,t^{2} + 54\,t^{3} + 676\,t + 5\,t^{4})\,
{\partial}^{2} \\
&\mbox{} + ( - 80\,t^{2} + 28\,t^{3} + 2\,t^{4} + 64 + 60\,t)\,
{\partial}^{3} \\
&\mbox{} + (58\,t^{2} - 88\,t - t^{4} - 12\,t^{3} + 12)\,{\partial}^{4}.
\end{align*}
Using the Maple package OreTools we check that $(\cK_1,\cK_2,\cK_3)=1$, then $A_{D_{\psi}}=A_{\psi}$ and $\c (A_{\psi})=0$.
Observe that there is no $\gamma\in\bbK$ such that $A_{\phi}= \gamma A_{\psi}$.

\vspace{0.5cm}{\bf \noindent Acknowledgements.}
The author thanks J. Rafael Sendra for kindly reading and commenting on some parts of this work.

\end{document}